\documentclass[a4paper,10pt,leqno]{amsart}
\usepackage{amsmath,amsthm,amssymb,latexsym,epsfig,graphicx,subfigure}
\numberwithin{equation}{section}

\newtheorem{thm}{Theorem}[section]

\newtheorem{lm}[thm]{Lemma}

\newtheorem{pr}[thm]{Proposition}
\theoremstyle{definition}

\theoremstyle{definition}

\newcommand{\Rn}{\mathbb{R}^{n}}

\newcommand{\R}{\mathbb{R}}

\newcommand {\grtrsim} {\ {\raise-.5ex\hbox{$\buildrel>\over\sim$}}\ }
\newcommand{\e}{\varepsilon}

\newcommand{\khii}{\text{\lower -.4ex\hbox{$\chi$}}}
\DeclareMathOperator{\spt}{spt}

\newcommand{\f}{\varphi}
\renewcommand{\a}{\alpha}

\begin{document}
\title {Hausdorff dimension and projections related to intersections}
\author{Pertti Mattila}


\begin{abstract} 
For $S_g(x,y)=x-g(y), x,y\in\Rn, g\in O(n),$ we investigate the Lebesgue measure and Hausdorff dimension of $S_g(A)$ given the dimension of $A$, both for general Borel subsets of $\R^{2n}$ and for product sets.
\end{abstract}

\maketitle

\section{introduction}

Let $A$ and $B$ be Borel subsets of $\Rn$. Under which conditions on the Hausdorff dimensions $\dim A$ and $\dim B$ do we have  
$A\cap (g(B)+z)\neq\emptyset$ for positively many, in the sense of Lebesgue measure $\mathcal L^n$, $z\in\Rn$ for almost all $g\in O(n)$? Defining  $S_{g}(x,y)=x-g(y)$, $A\cap (g(B)+z)\neq\emptyset$ for positively many $z\in\Rn$ is equivalent to  $\mathcal{L}^n(S_g(A\times B))>0$. Or we can also ask when $A\cap (g(B)+z)\neq\emptyset$ for $z$ in and $g$ outside a set of Hausdorff dimension of certain size. This reduces to estimating the dimension of $S_g(A\times B)$ and the dimension of the corresponding exceptional set of orthogonal transformations. 

In this paper we study more generally the Lebesgue measure and Hausdorff dimension of $S_g(A)$ for $A\subset\R^{2n}$. In Theorem \ref{thm2} we shall show for a Borel set $A\subset \R^{2n}$ that for almost all $g\in O(n)$,\ $\mathcal{L}^n(S_g(A))>0$, if $\dim A > n+1$, $\dim S_g(A)\geq \dim A -1$, if $n-1\leq\dim A\leq n+1$, and  $\dim S_g(A)\geq \dim A $, if $\dim A\leq n-1$. In all cases we also derive Hausdorff dimension estimates for the sets of exceptional $g\in O(n)$. In Theorems \ref{thm3} and \ref{thm4} we show that these estimates can be improved for product sets. We shall also comment on some relations to Falconer's distance set problem.

Instead of asking $A\cap (g(B)+z)$ to be non-empty, we could ask on the Hausdorff dimension of these intersections. This problem was studied in \cite{K}, \cite{M1}, \cite{M2}, \cite{M3}, \cite{M4}, \cite{M5} and \cite{M7}. I shall make comments on it at the end of the paper.  I expect the following to be true: if $A$ and $B$ are Borel subsets of $\Rn$ with $\dim A + \dim B > n$, then for almost all $g\in O(n)$,\ $\dim A\cap (g(B)+z)\geq \dim A + \dim B - n-\epsilon$ for every $\epsilon>0$ for positively many $z\in\Rn$. This is only known if one of the sets has dimension bigger than $(n+1)/2$. But in \cite{M8} we apply the results and methods of this paper to show that this also is true if $\dim A + (n-1)\dim B/n > n$ and $A$ and $B$ satisfy an extra condition of positive lower density.  

The family $S_g, g\in O(n)$, is a restricted family of orthogonal projections onto $n$-planes in $\R^{2n}$; it is only $n(n-1)/2$ dimensional while the full family of orthogonal projections has dimension $n^2$. Similar questions for other  restricted families of orthogonal projections have been studied by many people, see  \cite{JJLL},  \cite{JJK}, \cite{FO}, \cite{Or}, \cite{O}, \cite{OO}, \cite{KOV}, \cite{OV}. There also are discussions on these in \cite{M5} and \cite{M6}. 

Hausdorff dimension results for projections have their origin in Marstrand's projection theorem \cite{M}: for a Borel set $A\subset\R^2$, for almost all orthogonal projections $p$ onto lines,\ $\mathcal L^1(p(A))>0$, if $\dim A > 1$, and $\dim p(A) = \dim A$, if $\dim A \leq 1$. The study of exceptions was started by Kaufman \cite{Ka} who showed that in the second statement the dimension of the set of the exceptional projections is at most $\dim A$, and continued by Falconer \cite{F1} who showed that in the first statement the set of the exceptions has dimension at most $2-\dim A$. Discussion and further references can be found for example in \cite{M5}.

One could also study similar questions for other dimensions in place of Hausdorff dimension, for instance, Minkowski and packing dimensions. For these the situation is different. Examples of J\"arvenp\"a\"a in \cite{J1} show that there are sets with full dimension $n$ which project to measure zero on all $m$-planes. Hence the exact analogues of the intersection results of this paper are false even when one of the sets is a plane. However, 
J\"arvenp\"a\"a in \cite{J2} and \cite{J3} and Eswarathasan, Iosevich and Taylor in \cite{EIT} proved some related results.

I would like to thank the referees for many useful comments.

\section{Preliminaries}

We denote by $\mathcal L^n$ the Lebesgue measure in the Euclidean $n$-space $\Rn, n\geq 2,$ and by $\sigma^{n-1}$ the surface measure on the unit sphere $S^{n-1}$. The orthogonal group of $\Rn$ is $O(n)$ and its Haar probability measure is $\theta_n$. For $A\subset\Rn$ (or $A\subset O(n)$) we denote by $\mathcal M(A)$ the set of non-zero Radon measures $\mu$ on $\Rn$ with compact support $\spt\mu\subset A$. The Fourier transform of $\mu$ is defined by
$$\widehat{\mu}(x)=\int e^{-2\pi ix\cdot y}\,d\mu y,~ x\in\Rn.$$
We shall also use $\mathcal F$ to denote the Fourier transform.

For $0<s<n$ the $s$-energy of $\mu\in\mathcal M(\Rn)$ is 
\begin{equation}\label{eq10}
I_s(\mu)=\iint|x-y|^{-s}\,d\mu x\,d\mu y=c(n,s)\int|\widehat{\mu}(x)|^2|x|^{s-n}\,dx.
\end{equation} 
The second equality is a consequence of Parseval's formula and the fact that the distributional Fourier transform of the Riesz kernel $k_s, k_s(x)=|x|^{-s}$, is a constant multiple of $k_{n-s}$, see, for example, \cite{M4}, Lemma 12.12, or \cite{M5}, Theorem 3.10. These books contain most of the background material needed in this paper.

Notice that if $\mu$ satisfies the Frostman condition $\mu(B(x,r))\leq r^s$ for all $x\in\Rn, r>0$, then $I_t(\mu)<\infty$ for all $t<s$. We have for any Borel set $A\subset\Rn$ with $\dim A > 0$, cf. Theorem 8.9 in \cite{M4},
\begin{equation}\label{eq3}
\begin{split}
\dim A&=\sup\{s:\exists \mu\in\mathcal M(A)\ \text{such that}\ \mu(B(x,r))\leq r^s\ \text{for}\ x\in\Rn, r>0\}\\
&=\sup\{s:\exists \mu\in\mathcal M(A)\ \text{such that}\ I_s(\mu)<\infty\}.
\end{split}
\end{equation}

We shall denote by $f_{\#}\lambda$ the push-forward of a measure $\lambda$  under a map  $f: f_{\#}\lambda(A)= \lambda(f^{-1}(A))$.

By the notation $M\lesssim N$ we mean that $M\leq CN$ for some constant $C$. The dependence of $C$ should be clear from the context. The notation $M \approx N$ means that $M\lesssim N$ and $N\lesssim M$. By  $c$ we mean positive constants with obvious dependence on the related parameters. The closed ball with centre $x$ and radius $r$ will be denoted by $B(x,r)$. 

\section{Projections of general sets}

For $g\in O(n), t\in\R$, define
$$S_g, \pi_t:\Rn\times\Rn \to \Rn, S_g(x,y)=x-g(y), \pi_t(x,y)=x-ty.$$

Both of these can be realized as families of orthogonal projections. The family $S_g$ has curvature (in any natural sense) while $\pi_t$ does not have. See  \cite{FO}, \cite{Or} and \cite{KOV} for the role of curvature in projection theorems.

More precisely, let $\{e_1,\dots,e_n\}$ be an orthonormal basis for $\R^n$. Set $u_i=\frac{1}{\sqrt{2}}(e_i,-g^{-1}(e_i)), i=1,\dots,n.$ Then $\{u_1,\dots,u_n\}$ is an orthonormal basis for an $n$-plane $V_g\subset\R^{2n}$. The orthogonal complement of $V_g$, spanned by  $\frac{1}{\sqrt{2}}(e_i,g^{-1}(e_i)), i=1,\dots,n,$ is the kernel of $S_g$. Since $\frac{1}{\sqrt{2}}S_g(u_i)=e_i$, $\frac{1}{\sqrt{2}}S_g$ is essentially the orthogonal projection onto $V_g$. 

When $n=2$ we have in complex notation, $g$ identified with the angle $\phi$: $S_g(x,y)=x-e^{i\phi}y.$

Some relations between the projections $\pi_t$ and the Kakeya problem are discussed in \cite{M6}.

Recall the following lemma from \cite{M7}, Lemma 2.1. In \cite{M7} only the first bound was proven, but the second can be proven by an analogous argument as used for the first bound. Notice that the term $(n-1)(n-2)/2$ is needed there: the subgroup of $O(n)$ consisting of $(x,t)\mapsto (g(x),t), x\in\R^{n-1}, t\in\R, g\in O(n-1)$, has dimension $(n-1)(n-2)/2$ and $(g(0),1)=(0,1)$ for all $g\in O(n-1)$. 

\begin{lm}\label{lemma2} Let $\theta\in\mathcal M(O(n)), \a>(n-1)(n-2)/2$ and $\beta=\a-(n-1)(n-2)/2$. If $\theta(B(g,r))\leq r^{\a}$ for all $g\in O(n)$ and $r>0$, then for $x,z\in\Rn\setminus\{0\}, r>0$,
\begin{equation}\label{eq11}
\theta(\{g:|x-g(z)|< r\})\lesssim \min\{(r/|z|)^{\beta},(r/|x|)^{\beta}\}.
\end{equation}
\end{lm}

This will be applied via the following proposition, as in Chapter 5 of \cite{M5} and in many other places: 

\begin{pr}\label{prop}
Let $A\subset\Rn$ be a Borel set and $\beta>0, \gamma>0$. Suppose that  for any $\theta\in \mathcal M(O(n))$ such that \eqref{eq11} holds,  $\mathcal L^n(S_g(A))>0$ (or $\dim S_g(A)\geq \gamma$) for $\theta$ almost all $g\in O(n)$. Then there is a Borel set $E\subset O(n)$ such that $\dim E \leq \beta+(n-1)(n-2)/2$ and  $\mathcal L^n(S_g(A))>0$ (or $\dim S_g(A)\geq \gamma$) for $g\in O(n)\setminus E$.
\end{pr}

\begin{proof}
I skip the easy measurability arguments. If the Lebesgue measure part fails, the set $G$ of $g\in O(n)$ for which  $\mathcal L^n(S_g(A))=0$ has dimension greater than $\alpha=\beta+(n-1)(n-2)/2$. Then by \eqref{eq3}  there is $\theta\in \mathcal M(G)$ such that $\theta(B(g,r))\leq r^{\a}$ for all $g\in O(n)$ and $r>0$, so that \eqref{eq11} holds by Lemma \ref{lemma2}. By assumption, $\mathcal L^n(S_g(A))>0$  for $\theta$ almost all $g\in O(n)$, which contradicts the definition $G$ and that $\theta\in \mathcal M(G)$. The Hausdorff dimension part is proven by the same argument.
\end{proof}

The following theorem for $\pi_t$ essentially is a special case of Oberlin's results in \cite{O}. It was not explicitly stated there, but (1) and (2) follow by his arguments, see in particular the proof of Lemma 3.1 in \cite{O}. The proof of (3) is a standard argument of Kaufman from \cite{Ka}, see the proof of Theorem \ref{thm2}. The proof of Theorem \ref{thm2} also gives Theorem \ref{thm1} changing $g(x)$ to $tx$. 

\begin{thm}\label{thm1}
Let $A\subset\R^{2n}$  be a Borel set. 
\begin{itemize}
\item[(1)] If $\dim A > 2n-1$, then $\mathcal{L}^n(\pi_t(A))>0$ for $\mathcal L^1$ almost all $t\in\R$. Moreover, there is $E\subset\R$ such that $\dim E \leq 2n-\dim A$ and $\mathcal{L}^n(\pi_t(A))>0$ for  $t\in\R\setminus E$.
\item[(2)] If $n\leq\dim A \leq 2n-1$, then $\dim \pi_t(A) \geq \dim A -n+1$ for $\mathcal L^1$ almost all $t\in\R$. Moreover, for $\dim A - n \leq u\leq\dim A - n + 1$ there is $E\subset \R$ such that $\dim E \leq u+n-\dim A$ and  
$\dim \pi_t(A) \geq u$ for  $t\in\R\setminus E$. 
\item[(3)] If $\dim A \leq n$, then $\dim \pi_t(A) \geq \min\{\dim A,1\}$ for $\mathcal L^1$ almost all $t\in\R$. Moreover, for $0<u\leq\min\{\dim A,1\}$ there is $E\subset \R$ such that $\dim E \leq u$ and    
$\dim \pi_t(A) \geq u$ for  $t\in\R\setminus E$. 
\item[(4)] For all $t\in\R$,\ $\dim \pi_t(A) \geq \dim A -n$.
\end{itemize}
\end{thm}

Notice that the last statement is trivial, because associating with $\pi_t$ an orthogonal projection $p_t$, as for $S_g$ in the beginning of this section, $A\subset p_t(A)\times p_t^{-1}(0)$ and 
$\dim(p_t(A)\times p_t^{-1}(0))=\dim p_t(A)+n=\dim\pi_t(A)+n$.

This theorem is valid also when $n=1$; it is Marstrand's projection with Kaufman's and Falconer's exceptional set estimates.

We have a similar result for $S_g$. Observe also there that (4) is trivial. The proof below for (1) and (2) is a modification of Oberlin's proof. The proof of (3) again is Kaufman's argument.

\begin{thm}\label{thm2}
Let $A\subset\R^{2n}$  be a Borel set. \\
\begin{itemize}
\item[(1)] If $\dim A > n+1$, then $\mathcal{L}^n(S_g(A))>0$ for $\theta_{n}$ almost all $g\in O(n)$. Moreover, there is $E\subset O(n$) such that $\dim E \leq 2n-\dim A+(n-1)(n-2)/2$ and $\mathcal{L}^n(S_g(A))>0$ for  $g\in O(n)\setminus E$.
\item[(2)] If $n-1\leq \dim A \leq n+1$, then $\dim S_g(A) \geq \dim A -1$ for $\theta_{n}$ almost all $g\in O(n)$. Moreover, for any $\dim A - n \leq u\leq\dim  A-1$ there is $E\subset O(n$) such that $\dim E \leq u+n-\dim A+(n-1)(n-2)/2$ and $\dim S_g(A) \geq u$ for  $g\in O(n)\setminus E$.
\item[(3)] If $\dim A \leq n-1$, then $\dim S_g(A) \geq \dim A$ for $\theta_{n}$ almost all $g\in O(n)$. Moreover, for $0<u\leq\dim A$ there is $E\subset O(n$) such that $\dim E \leq u+(n-1)(n-2)/2$ and $\dim S_g(A) \geq u$ for  $g\in O(n)\setminus E$.
\item[(4)] For all $g\in O(n)$,\ $\dim S_g(A) \geq \dim A -n$
\end{itemize}
\end{thm}

\begin{proof}
Let $0<s<\dim A$ and  $\mu\in\mathcal M(A)$ with $I_s(\mu) < \infty$.

Let $\mu_g\in\mathcal M(S_g(A))$ be the push-forward of $\mu$ under $S_g$. Then for $\xi\in\Rn$,
\begin{align*}
&\widehat{\mu_g}(\xi)=\int e^{-2\pi i\xi\cdot S_g(x,y)}\,d\mu (x,y)\\ 
&=\int e^{-2\pi i(\xi,-g^{-1}(\xi))\cdot(x,y)}\,d\mu (x,y)=\widehat{\mu}(\xi,-g^{-1}(\xi)).
\end{align*}
Let $0<\beta\leq n-1$ and let $\theta\in\mathcal M(O(n))$ be such that  for $x,z\in\Rn\setminus\{0\}, r>0,$ 
\begin{equation}\label{eq9}
\theta(\{g\in O(n):|x-g(z)|<r\})\leq \min\{(r/|z|)^{\beta},(r/|x|)^{\beta}\}.
\end{equation}

To prove (1) and (2) we shall show that for $R>1$,
\begin{equation}\label{obeq}
\iint_{R\leq|\xi|\leq 2R}|\widehat{\mu}(\xi,-g^{-1}(\xi))|^2\,d\xi\,d\theta g\lesssim R^{2n-s-\beta}.
\end{equation}

This is applied to the dyadic annuli, $R=2^k, k=1,2,\dots$. The sum converges if $s>2n-\beta$, and we can choose $\mu$ with such $s$ if $\dim A > 2n-\beta$. This gives $\iint|\widehat{\mu_g}(\xi)|^2\,d\xi\,d\theta g<\infty$. Hence for $\theta$ almost all $g\in O(n)$,\ $\mu_g$ is absolutely continuous with $L^2$ density, and so $\mathcal L^n(S_g(A))>0$. Taking $\beta=n-1$ and $\theta=\theta_n$, we get the first part of (1). The second follows with general $\beta$ and $\theta$ using Proposition \ref{prop}. 

To prove part (2) let $0<u<s+\beta-n$ and $\mu$ as above. Then \eqref{obeq} yields
\begin{equation*}
\iint|\widehat{\mu_g}(\xi)|^2|\xi|^{u-n}\,d\xi\,d\theta g<\infty,
\end{equation*}
so by \eqref{eq10} and \eqref{eq3}, $\dim S_g(A)\geq u$ for $\theta$ almost all $g\in O(n)$ and thus (2) follows with the same argument as above.

Now we begin the proof of \eqref{obeq}. From \eqref{eq9} we get for $\xi,y\in\R^n, R\leq |\xi| \leq 2R, M>\beta,$

\begin{equation}\label{theta}
\int(1+|\xi+g(y)|)^{-M}\,d\theta g \lesssim R^{-\beta},
\end{equation}

because 

\begin{align*}
&\int(1+|\xi+g(y)|)^{-M}\,d\theta g\\
&\leq \theta(\{g\in O(n):|\xi+g(y)| < 1\}) + \int_{\{g:|\xi+g(y)|\geq 1\}}(1+|\xi+g(y)|)^{-M}\,d\theta g\\
&\lesssim R^{-\beta} + \sum_{j= 0}^{\infty}2^{-Mj}\theta(\{g\in O(n):2^j\leq |\xi+g(y)| < 2^{j+1}\})\\
&\lesssim R^{-\beta} + \sum_{j= 0}^{\infty}2^{-Mj}(2^{j}/|\xi|)^{\beta}\lesssim R^{-\beta}.
\end{align*}

Choose a smooth compactly supported function $\phi$ which equals 1 on the support of $\mu$. Then $\widehat{\mu}=\widehat{\phi\mu}=\widehat{\phi}\ast\widehat{\mu}$ and the integral in \eqref{obeq} equals

\begin{align*}
I_R&:=\iint_{R\leq|\xi|\leq 2R}|\widehat{\phi\mu}(\xi,-g^{-1}(\xi))|^2\,d\xi\,d\theta g\\
&=\iint_{R\leq|\xi|\leq 2R}\left|\int\widehat{\phi}((\xi,-g^{-1}(\xi))-y)\widehat{\mu}(y)\,dy\right|^2\,d\xi\,d\theta g.
\end{align*}
By the Schwartz inequality,
\begin{align*}
&I_R\leq\iint_{R\leq|\xi|\leq 2R}(\int|\widehat{\phi}((\xi,-g^{-1}(\xi))-y)|\,dy\\ &\int|\widehat{\phi}((\xi,-g^{-1}(\xi))-y)||\widehat{\mu}(y)|^2\,dy)\,d\xi\,d\theta g\\
&\lesssim \iint_{R\leq|\xi|\leq 2R}\int|\widehat{\phi}((\xi,-g^{-1}(\xi))-y)||\widehat{\mu}(y)|^2\,dy\,d\xi\,d\theta g\\
&\lesssim \iint_{R\leq|\xi|\leq 2R}\int(1+|(\xi,-g^{-1}(\xi))-y|)^{-3M}|\widehat{\mu}(y)|^2\,dy\,d\xi\,d\theta g,
\end{align*}
by the fast decay of $\widehat{\phi}$, where  $M>2n$. Clearly, with $y=(y_1,y_2), y_1,y_2\in\Rn$, 
$$|(\xi,-g^{-1}(\xi))-y)|\geq \max\{|\xi-y_1|,|\xi+g(y_2)|\}.$$
Moreover, $|(\xi,-g^{-1}(\xi))-y)|\approx |y|$, when $R\leq|\xi|\leq 2R$ and $|y|>5R$.
Hence
\begin{align*}
&I_R\lesssim\\
& \int_{|y|\leq 5R} \int_{R\leq|\xi|\leq 2R}\int(1+|\xi+g(y_2)|)^{-M}\,d\theta g(1+|\xi-y_1|)^{-M}\,d\xi|\widehat{\mu}(y)|^2\,dy\\
&+\int_{|y| > 5R}\int_{R\leq|\xi|\leq 2R}\int(1+|\xi+g(y_2)|)^{-M}\,d\theta g(1+|\xi-y_1|)^{-M}\,d\xi|y|^{-M}\,dy.
\end{align*}

We have by \eqref{theta}
$$\int(1+|\xi+g(y_2)|)^{-M}\,d\theta g \lesssim R^{-\beta}.$$ 
Since $\int (1+|\xi-y_1|)^{-M}\,d\xi$ is bounded, we obtain
$$I_R\lesssim
R^{-\beta}\left(\int_{|y|\leq 5R}|\widehat{\mu}(y)|^2\,dy+\int_{|y| > 5R} |y|^{-M}\,dy\right).$$
The second integral is bounded and for the first we have by \cite{M5}, Section 3.8,
\begin{equation}\label{eq8}
\int_{|y|\leq 5R}|\widehat{\mu}(y)|^2\,dy \lesssim R^{2n-s},
\end{equation}
which imply $I_R\lesssim R^{2n-s-\beta}$ as required. Thus we have proved \eqref{obeq}, hence also (1) and (2).

To prove (3) suppose that $0<u<\dim A\leq n-1, \mu\in\mathcal M(A)$ with $I_u(\mu) < \infty$ and let $\theta$ and $\beta$ be as in \eqref{eq9} with $\beta > u$. It suffices to show that $\dim S_g(A) \geq u$ for $\theta$ almost all $g\in O(n)$. Using \eqref{eq9} this follows from

\begin{align*}
&\int I_u(\mu_g)\,d\theta g = \iiint |x-y|^{-u}\,dS_{g\#}\mu x\,dS_{g\#}\mu y\,d\theta g\\
&=\iiint |S_g(w-z)|^{-u}\,d\mu w\,d\mu z\,d\theta g\\
&=\iiint_0^{\infty} \theta(\{g:|S_g(w-z)|^{-u}>r\})\,dr\,d\mu w\,d\mu z\\
&=\iiint_0^{\infty} \theta(\{g:|S_g(w-z)|<r^{-1/u}\})\,dr\,d\mu w\,d\mu z\\
&\lesssim\iiint_0^{|w-z|^{-u}} \,dr\,d\mu w\,d\mu z + \iiint_{|w-z|^{-u}}^{\infty} (r^{-1/u}/|w-z|)^{\beta}\,dr\,d\mu w\,d\mu z\\
&\approx  I_{u}(\mu)<\infty,
\end{align*}
which completes the proof of the theorem.
\end{proof}

\subsection{Sharpness}
The bounds in the $\mathcal L^1$ almost all statements of Theorem \ref{thm1} are sharp when $n=2$. To see this let $0\leq s \leq 1, C_s\subset\R$ with $\dim C_s=s$, and $A_s = \{(x,y)\in\R^2\times\R^2:x_1\in C_s, y_1=0\}.$ Then $\dim A_s = 2+s, \pi_t(A_s) =  C_s\times\R$ and $\dim \pi_t(A_s) = 1+s$. This shows that (2) is sharp. For (1) we can choose $C_1$ with $\mathcal L^1(C_1)=0$, then $\dim A_1=3$ and $\mathcal L^2(\pi_t(A))=0$. 
If $1\leq\dim A\leq 2$ we can only say that $\dim\pi_t(A)\geq 1$ for almost all $t\in\R$ since $\pi_t(\R\times\{0\}\times\R\times\{0\})=\R\times\{0\}$. Hence (3) also is sharp. Probably the bounds for the dimensions of the exceptional sets are sharp, too. Perhaps this could be seen using similar examples as in \cite{KM}, see also Example 5.13 in \cite{M5}, but I have not checked it.

When $n\geq 3$ a similar argument shows that the $\mathcal L^1$ almost all statements of Theorem \ref{thm1} are sharp when $\dim A\geq 2n-2$ or $\dim A \leq 2$. Probably it is not sharp in the remaining ranges.

I don't know if the bounds are sharp for $S_g$. In Section \ref{product sets} we shall see that some of them can be improved for product sets. By the above examples this is not possible for $\pi_t$. 

I illustrate the role of $\dim O(n-1)=(n-1)(n-2)/2$ with two simple examples:
For $0<s\leq 1$ choose a compact set $C_s\subset\R$ such that $\dim C_s = \dim(C_s-C_s)=s, \dim (C_s\times C_s)=2s$ and $\mathcal L^1(C_1-C_1)=0$. Such sets are easy to construct. If $A_s=\R^{n-1}\times C_s\times\R^{n-1}\times C_s$, then $\dim A_s = 2s + 2(n-1), S_g(A)=\R^{n-1}\times (C_s-C_s)$ and $\dim S_g(A_s) = s + n-1$  for $g\in O(n-1)$ (identified with $(x,t)\mapsto (g(x),t))$. In particular, $\dim A_1= 2n$ and $\mathcal L^n(S_g(A_1))=0$ for $g\in O(n-1)$. Next, take $B_s=\{0\}\times C_s\times\{0\}\times C_s$. Then 
$\dim B_s = 2s, S_g(A)=\{0\}\times (C_s-C_s)$ and $\dim S_g(B_s) = s$  for $g\in O(n-1)$.

\subsection{An alternative argument}\label{alter}
Here is another simple argument for the statement 'If $A\subset\R^{2n}$ is a Borel set and $\dim A > n+1$, then $\mathcal{L}^n(S_g(A))>0$ for $\theta_{n}$ almost all $g\in O(n)$':

Let   $\mu\in\mathcal M(A)$ with $I_{n+1}(\mu)<\infty$. Consider for $r>0$, 

\begin{align*}
&I_{r}=r^{-n}\iint S_{g\#}\mu(B(z,r))\,dS_{g\#}\mu z\,d\theta_ng\\
&=r^{-n}\iint\theta_n(\{g:|x-u-g(y-v)|\leq r\})\,d\mu (u,v)\,d\mu(x,y)\\
&\lesssim r^{-1}\iint_{\{(x,y):||x-u|-|y-v||\leq r\}}|y-v|^{1-n}\,d\mu(u,v)\,d\mu (x,y).\\
\end{align*}

Let $\phi\in C_0^{\infty}(\Rn)$ with $\phi(y)=1$ when $(x,y)\in spt\mu$ for some $x$, and let
$$\psi_{r}(x,y)=\chi_{\{(x,y):||x|-|y||\leq r\}}(x,y)|y|^{1-n}\phi(y).$$ 
Then

$$I_{r}\lesssim r^{-1}\int\psi_{r}\ast\mu\,d\mu=r^{-1}\int\widehat{\psi_{r}}|\widehat{\mu}|^2.$$
Let $\sigma_r$ be the surface measure on $\{x\in\Rn: |x|=r\}$. Then for any $u,y\in\R^n$, $\widehat{\sigma_{|y|}}(u)=|y|^{n-1}|u|^{1-n}\widehat{\sigma_{|u|}}(y)$. Thus for small $r$,
\begin{align*}
|r^{-1}\widehat{\psi_{r}}(u,v)|&=\left|r^{-1}\iint_{||x|-|y||\leq r}|y|^{1-n}\phi(y)e^{-2\pi i(u\cdot x+v\cdot y)}\,dx\,dy\right|\\
&\approx\left|\int|y|^{1-n}\phi(y)\widehat{\sigma_{|y|}}(u)e^{-2\pi iv\cdot y}\,dy\right|\\
&=\left|\int|y|^{1-n}\phi(y)|y|^{n-1}|u|^{1-n}\widehat{\sigma_{|u|}}(y)e^{-2\pi iv\cdot y}\,dy\right|\\
&=\left||u|^{1-n}\mathcal F(\phi\widehat{\sigma_{|u|}})(v)\right|=\left||u|^{1-n}\int\widehat{\phi}(y-v)\,d\sigma_{|u|}y\right|\\
&\lesssim |u|^{1-n}(1+||u|-|v||)^{1-n}\lesssim|(u,v)|^{1-n},\\
\end{align*}
the second to last by the fast decay of $\widehat{\phi}$. Hence
$$I_{r}\lesssim\iint|(u,v)|^{1-n}|\widehat{\mu}(u,v)|^2\,d(u,v)\approx I_{n+1}(\mu).$$
Define the lower derivative, with $\a(n)=\mathcal L^n(B(0,1))$,
$$\underline{D}(S_{g\#}\mu)(z)=\liminf_{r\to 0}\a(n)^{-1}r^{-n}S_{g\#}\mu(B(z,r)).$$
Letting $r\to 0$ and using Fatou's lemma we then see that 
\begin{equation}\label{eq5}
\iint \underline{D}(S_{g\#})(z)^2\,dz\,d\theta_{n}g=\int \underline{D}(S_{g\#})(z)\,dS_{g\#}\mu z\,d\theta_{n}g<\infty,
\end{equation}
which implies that $S_{g\#}\mu<< \mathcal{L}^n$ with $L^2$ density, see e.g. \cite{M4}, Theorem 2.12, for $\theta_{n}$ almost all $g$, from which the claim follows.\\ \\

\section{Product sets}\label{product sets}

 For product sets we can improve Theorem \ref{thm2} for $S_g$, but not for $\pi_t$, as the previous examples show. Let $\theta\in\mathcal M(S^{n-1})$ and $0<\beta \leq n-1$.  Suppose that for $x,z\in\Rn, r>1$,
\begin{equation}\label{eq12}
\theta(\{g:|x-g(z)|< r\})\lesssim (r/|x|)^{\beta}.
\end{equation}

Let $\mu\in\mathcal M(\R^n)$ and set for $r>1$ and $\xi\in\Rn$,

$$\sigma(\mu)(r)=\int_{S^{n-1}}|\widehat{\mu}(rv)|^2\,d\sigma^{n-1}v,$$ 
$$\sigma_{\theta}(\mu)(\xi)=\int|\widehat{\mu}(g^{-1}(\xi))|^2\,d\theta g.$$ 

Then $\sigma_{\theta_{n}}(\mu)(\xi)=c\sigma(\mu)(|\xi|)$. 

The decay estimates for $\sigma(\mu)(r)$ have been studied by many people, a discussion can be found in \cite{M5}. The  best known estimates, due to Wolff, \cite{W}, when $n=2$, and to Du and Zhang , \cite{DZ}, in the general case, are the following: Let  $\mu\in\mathcal M(\Rn)$ with $\mu(B(x,r))\leq r^s$ for $x\in\R^n, r>0$. Then
for all $\epsilon > 0, r>1$, 

\begin{equation}\label{dz}
\sigma(\mu)(r) \lesssim \begin{cases} r^{-(n-1)s/n+\epsilon}\ \text{for all}\ 0<s<n,\\
r^{-(n-1)/2+\epsilon}\ \text{if}\ 0<(n-1)/2\leq s\leq n/2,\\
 r^{-s+\epsilon}\ \text{if}\ 0<s\leq (n-1)/2.\\
\end{cases}
\end{equation}
The middle estimate is kind of unnecessary to state since it follows from the last. I shall drop the corresponding case later on. The essential case for the first estimate is $s>n/2$, otherwise the second and third estimates are better. Up to $\epsilon$ these estimates are sharp when $n=2$. When $n\geq 3$ the sharp bounds are not known for all $s$, see \cite{D} for discussion and the most recent examples.

For $r>1$, let
$$A_r=\{x\in\R^n:r-1<|x|<r+1\}.$$ 
It is easy to see that for large  $r$,\ $\sigma(\mu)(r)\lesssim r^{-\alpha + \epsilon}$ for all $\epsilon>0$ if and only if $r^{1-n}\int_{A_r}|\widehat{\mu}(x)|^2\,dx \lesssim r^{-\alpha + \epsilon}$ for all $\epsilon>0$. Indeed, the implication from left to right is trivial. The opposite implication  is Proposition 16.2 in \cite{M5}.  This and the above estimates for 
$\sigma(\mu)(r)$ yield the following lemma:

\begin{lm}\label{Lemma1} Let $\theta\in\mathcal M(S^{n-1})$ satisfy \eqref{eq12} with exponent $0<\beta \leq n-1$. If $\mu(B(x,r))\leq r^s$ for $x\in\R^n, r>0$, then for every $\xi\in \R^n$  with $|\xi|>1$ and for every $\epsilon>0$,
\begin{equation*}
\sigma_{\theta}(\mu)(\xi)\lesssim \begin{cases} |\xi|^{-(n-1)s/n+n-1-\beta+\epsilon}\ \text{for all}\ 0<s<n,\\
 |\xi|^{-s+n-1-\beta+\epsilon}\ \text{if}\ 0<s\leq (n-1)/2.\\
\end{cases}
\end{equation*}
\end{lm}

\begin{proof}
I only consider the first estimate, the second follows in the same way. Using the above estimate for $\sigma(\mu)(r)$ and the above mentioned relation to the estimates over the annuli $A_r$, we have 
\begin{equation}\label{annulus}
r^{1-n}\int_{A_r}|\widehat{\mu}(x)|^2\,dx \lesssim r^{-(n-1)s/n+\epsilon}\end{equation}
for all $\epsilon>0$. The proof of Proposition 16.2 in \cite{M5} works for $\theta$ in place of $\theta_{n}$ as such yielding the required estimate. I give a short sketch. 

Choose a smooth function $\f, \f\geq 0,$ such that $\f = 1$ on $\spt \mu$.  Then arguments similar to those in the proof of Theorem \ref{thm2} show that

\begin{equation*}
\begin{split}
\sigma_{\theta}(\mu)(\xi) 
& \lesssim  \iint_{ \{(x,g) : |g^{-1}(\xi)-x|<|\xi|^\e \}} \,d \theta g |\widehat{\mu} (x)|^2 \,d x\\
&+\sum_{j=1}^{\infty} \iint_{\{(x,g) :  (|\xi|^\e)^j\leq|g^{-1}(\xi)-x|<(|\xi|^\e)^{j+1} \}}|\widehat{\f}(g^{-1}(\xi)-x)||\widehat{\mu} (x)|^2 \,dx\, d\theta g.
\end{split}
\end{equation*}
The second term is easily handled by \eqref{annulus} and the fast decay of $\f$. 

We have by \eqref{eq12}

$$\theta(\{g:|g^{-1}(\xi)-x|<|\xi|^{\e}\}) \lesssim |\xi|^{(\e-1)\beta},$$
whence the first term is bounded by $$|\xi|^{(\e-1)\beta}\int_{\{x:|\xi|-|\xi|^{\e}<|x|<|\xi|+|\xi|^{\e}\}}|\widehat{\mu}(x)|^2\,dx.$$
Covering the interval $(|\xi|-|\xi|^{\e},|\xi|+|\xi|^{\e})$ with about $|\xi|^{\e}$ intervals of length $2$ and applying \eqref{annulus} we obtain the required bound.
\end{proof}

\begin{thm}\label{thm3}
Let $A, B\subset\R^{n}$  be Borel sets.
\begin{itemize}
\item[(1)]Suppose $\dim A+\dim B>n$. If $\dim A + (n-1)\dim B/n > n$ or $\dim A > (n+1)/2$, then $\mathcal{L}^2(S_g(A\times B))>0$ for $\theta_{n}$ almost all $g\in O(n)$.
\item[(2)] If $\dim A + (n-1)\dim B/n\leq n$, then \\
$\dim S_g(A\times B) \geq \dim A + (n-1)\dim B/n$ for $\theta_{n}$ almost all $g\in O(n)$. \\
If $\dim A + \dim B\leq n$ and $\dim B\leq (n-1)/2$, then\\
$\dim S_g(A\times B) \geq \dim A + \dim B$ for $\theta_{n}$ almost all $g\in O(n)$.\\
\end{itemize}
\end{thm}

We have the following exceptional set estimates:

\begin{thm}\label{thm4}
Let $A, B\subset\R^{n}$  be Borel sets.
\begin{itemize}
\item[(1)]Suppose $\dim A+\dim B>n$. Then there is $E\subset O(n)$ such that $\mathcal{L}^2(S_g(A\times B))>0$ for $g\in O(n)\setminus E$ and\\
$\dim E \leq 2n-1-\dim A - (n-1)\dim B/n+(n-1)(n-2)/2$.\\
Moreover,\\
$\dim E \leq 2n-1-\dim A - \dim B+(n-1)(n-2)/2$, if $\dim B\leq (n-1)/2,$\\
\item[(2)] Let $0<\a\leq n$. Then there is $E\subset O(n)$ such that $\dim S_g(A\times B)\geq \a$ for $g\in O(n)\setminus E$ and\\
$\dim E \leq \alpha+n-1-\dim A - (n-1)\dim B/n+(n-1)(n-2)/2$. \\
Moreover,\\
$\dim E \leq \alpha+n-1-\dim A - \dim B+(n-1)(n-2)/2$, if $\dim B \leq (n-1)/2$.
\end{itemize}
\end{thm}

Notice that in some cases the upper bound for $\dim E$ is bigger than $n-1+(n-1)(n-2)/2=n(n-1)/2=\dim O(n)$. Then we can take $E=O(n)$ and the statement is empty.

\begin{proof}[Proofs of Theorems \ref{thm3} and \ref{thm4}]
The case  $\dim A > (n+1)/2$ in the first part of (1) of Theorem \ref{thm3} follows from Lemma 13.9 in \cite{M4} or from Lemma 7.1 in \cite{M5}.  I don't know any exceptional set estimates under the condition $\dim A > (n+1)/2$. 

Let $0<s<\dim A$ and $0<t<\dim B$ and let $\mu\in\mathcal M(A),\nu\in\mathcal M(B)$ with $\mu(B(x,r))\leq r^{s'}, \nu(B(x,r))\leq r^{t'}$ for some $s'>s, t'>t$ and  for $x\in\R^n, r>0$. Let $\lambda_g = S_{g\#}(\mu\times\nu)\in\mathcal M(S_g(A\times B)).$ Then $\widehat{\lambda_g}(\xi)=\widehat{\mu}(\xi)\widehat{\nu}(-g^{-1}(\xi))$. For $0<\alpha\leq n$ we have by Lemma \ref{Lemma1}
\begin{equation}\label{eq4}
\begin{split}
&\iint|\widehat{\lambda_g}(\xi)|^2|\xi|^{\alpha-n}\,d\xi\,d\theta g\\
 &=\int\sigma_{\theta}(\nu)(-\xi)|\widehat{\mu}(\xi)|^2|\xi|^{\alpha-n}\,d\xi\\
&\lesssim\int|\widehat{\mu}(\xi)|^2|\xi|^{\alpha-1-(n-1)t/n-\beta}\,d\xi\\
&=cI_{\alpha+n-1-(n-1)t/n-\beta}(\mu)\lesssim I_s(\mu)<\infty,
\end{split}
\end{equation}
if $\beta \geq \alpha + n-1-(n-1)t/n-s$.  

Similarly, if $t\leq (n-1)/2$,

\begin{equation}\label{eq7}
\iint|\widehat{\lambda_g}(\xi)|^2|\xi|^{\alpha-n}\,d\xi\,d\theta g
\lesssim I_{\alpha+n-1-t-\beta}(\nu)\lesssim I_s(\mu)<\infty,
\end{equation}
if $\beta \geq \alpha + n-1-s-t$.  

To get (1) of Theorems \ref{thm3} and \ref{thm4} we take $\alpha=n$. If $\beta \geq 2n-1-(n-1)t/n-s$, we have $S_{g\#}(\mu\times\nu)<<\mathcal L^n$, and so $\mathcal L^n(S_g(A\times B))>0$, for $\theta$ almost all $g\in O(n)$. In the case $\dim A + (n-1)\dim B/n > n$ we can choose $s$ and $t$ so that $n-1 \geq 2n-1-(n-1)t/n-s$. Then we can take $\theta = \theta_{n}$ and $\beta=n-1$ to get (1) of Theorem \ref{thm3} in this case.  For Theorem \ref{thm4}(1)  we have $\mathcal L^n(S_g(A\times B))>0$ for $\theta$ almost all $g\in O(n)$ provided  $\beta \geq 2n-1-(n-1)t/n-s$. Using Proposition \ref{prop} we see from this that the set of $g\in  O(n)$ for which $\mathcal L^n(S_g(A\times B))=0$ has dimension at most $2n-1-\dim A - (n-1)\dim B/n+(n-1)(n-2)/2$. The case $\dim B \leq (n-1)/2$ follows in the same way using \eqref{eq7}.

For any $0<\alpha\leq n$ we have that if $\beta \geq \alpha + n-1-(n-1)t/n-s$, then by \eqref{eq4} $I_{\alpha}(S_{g\#}(\mu\times\nu))<\infty$, and so $\dim S_g(A\times B) \geq \alpha$, for $\theta$ almost all $g\in O(n)$. To get the first statement of (2) of Theorem \ref{thm3} we take $\alpha = (n-1)t/n+s$ and $\beta = n-1$.  The case $\dim B \leq (n-1)/2$ of Theorem \ref{thm3}(2) follows in the same way. For Theorem \ref{thm4}(2)  we use Proposition \ref{prop} as before. 

\end{proof}

\section{Further discussion and open problems}

\subsection{Averages over a cone}\label{cone}
When $\theta=\theta_n$ we have for the integral in \eqref{obeq} 
\begin{equation*}
\iint_{R\leq|\xi|\leq 2R}|\widehat{\mu}(\xi,g(\xi))|^2\,d\xi\,d\theta_{n} g = R^n\int_{1\leq |x|=|y|\leq 2}|\widehat{\mu}(Rx,Ry)|^2\,d\gamma(x,y),
\end{equation*}
where the integration on the right side is with respect to a suitably normalized surface measure $\gamma$ on the conical surface $\Gamma = \{(x,y)\in\Rn\times \Rn: 1\leq |x|=|y|\leq 2\}$. Let  $\phi$ be a smooth non-negative function with compact support in $\{y:1/2<|y|<3\}$ and with $\phi(y) = 1$ when  $1\leq |y| \leq 2$. 
Define the measure $\lambda$ by
$$\int f\,d\lambda = \iint_{|x|=|y|}f(x,y)\,d\sigma_{|y|}x\phi(y)\,dy.$$
Then $\gamma\lesssim\lambda$.

The Fourier transform of $\lambda$ has the estimate

\begin{equation}\label{gamma}
|\widehat{\lambda}(\xi)| \lesssim |\xi|^{1-n},
\end{equation}

because

\begin{align*}
|\widehat{\lambda}(u,v)| &= \int e^{-2\pi i(u\cdot x + v\cdot y)}\,d\lambda(x,y)\\ 
& =c\int e^{-2\pi i v\cdot y}\widehat{\sigma_{|y|}}(u)\phi(y)\,dy\\ 
&= c\int e^{-2\pi i v\cdot y}|y|^{n-1}|u|^{1-n}\widehat{\sigma_{|u|}}(y)\phi(y)\,dy\\ 
&= c|u|^{1-n}\mathcal F(|y|^{n-1}\phi(y)\widehat{\sigma_{|u|}}(y))(v)\\ 
&= c|u|^{1-n}\int\mathcal F(|y|^{n-1}\phi(y))(v-x)\,d\sigma_{|u|}x\approx |(u,v)|^{1-n},\\ 
\end{align*}
where the last estimate follows as for the Fourier transform of $\psi_{r}$ above. Let $\mu\in\mathcal M(\R^{2n})$ with $I_s(\mu)<\infty$. 
Then using a general theorem of Erdo\u gan, Theorem 1 in \cite{E}, we obtain for $R>1$,

\begin{equation}\label{eq6}
\begin{split}
&\iint_{1\leq |x|=|y|\leq 2}|\widehat{\mu}(Rx,Ry)|^2\,d\gamma(x,y)\\
&\lesssim\iint|\widehat{\mu}(Rx,Ry)|^2\,d\lambda(x,y)\lesssim
\begin{cases}
&R^{1-s}\ \text{for all}\ 0<s<2n,\\
&R^{-s}\ \text{if}\ 0<s\leq n-1.
\end{cases} 
\end{split}
\end{equation}

We shall not use these estimates in this paper. For Theorem \ref{thm2} they give another proof for the almost all statements  with respect to $\theta_n$, but they do not improve Theorem \ref{thm2} and they don't give the exceptional set estimates. In fact, the second estimate in \eqref{eq6} is the same as \eqref{obeq} with $\theta = \theta_{n}, \beta = n-1.$ Better decay estimates for \eqref{eq6} might lead to improvements for Theorem \ref{thm2}. In particular, any improvement of the exponent $1-s$ in the range $n<s<n+1$ would lead to an improvement of the first statement of Theorem \ref{thm2}(1). I am not aware of such results. However, in addition to the spherical averages (discussed in Section \ref{product sets}) which have been studied for a long time, there are recent estimates for cones and hyperboloids, see \cite{CHL}, \cite{H} and \cite{BEH}. 

The estimates \eqref{eq6} can be improved for product measures. We just plug in the spherical estimates from \eqref{dz}. Let $\mu,\nu\in\mathcal M(\Rn)$ be such that for some $0<s\leq n$ and $0<t\leq n$ we have $\mu(B(x,r)) \leq r^s$ and  $\nu(B(x,r)) \leq r^t$ for $x\in\Rn, r>0$.  Then for all $\epsilon > 0, R>1$, 
\begin{equation*}
\iint_{1\leq |x|=|y|\leq 2}|\widehat{\mu\times\nu}(Rx,Ry)|^2\,d\gamma(x,y)\lesssim 
\begin{cases} R^{-s - (n-1)t/n + \epsilon}\ \text{if}\ 0<s<n,\\
 R^{-s -t + \epsilon}\ \text{if}\ 0<t\leq( n-1)/2.\\
\end{cases}
\end{equation*}
To see this note that $\widehat{\mu\times\nu}(x,y) = \widehat{\mu}(x)\widehat{\nu}(y)$. Then if $\sigma(\nu)(r)\lesssim r^{-\a}$, we get by \eqref{eq8}
\begin{align*}
&\iint_{1\leq |x|=|y|\leq 2}|\widehat{\mu\times\nu}(Rx,Ry)|^2\,d\gamma(x,y)\\
&= c\int_{1\leq |x|\leq 2}\sigma(\nu)(R|x|)|\widehat{\mu}(Rx)|^2\,dx\\
&\lesssim  R^{-\a}\int_{1\leq |x|\leq 2}|\widehat{\mu}(Rx)|^2\,dx
\lesssim R^{-\a - s},
\end{align*}
and the claims follow from \eqref{dz}.

\subsection{Distance sets and measures}
There are some connections of this topic to Falconer's distance set problem. For general discussion and references, see for example \cite{M5}. Falconer showed in \cite{F2} that for a Borel set $A\subset\Rn$ the distance set $\{|x-y|: x,y\in A\}$ has positive Lebesgue measure if $\dim A > (n+1)/2$. We had the same condition in Theorem \ref{thm3}  and it appeared in the intersection results of \cite{M2}. When $n=2$ Wolff \cite{W} improved $3/2$ to $4/3$.  Observe that when $\dim A = \dim B$, the assumption $\dim A + \dim B/2 > 2$ in Theorem \ref{thm3} becomes $\dim A > 4/3$ and is the same as Wolff's. In \cite{IL} Iosevich and Liu improved distance set results of the time for product sets with rather simple arguments. For the most recent,  and so far the best known, distance set results, see \cite{GIOW}, \cite{DGOWWZ} and \cite{DIOWZ}. 

The proofs of distance set results often involve the distance measure $\delta(\mu)$ of a measure $\mu$ defined by

$$\delta(\mu)(B) = \mu\times\mu(\{(x,y): |x-y|\in B\}),\ B\subset\R.$$

For example, Wolff showed that $\delta(\mu)\in L^2(\R)$, if $I_{s}(\mu)<\infty$ for some $s>4/3$. To do this he used decay estimates for the spherical averages $\sigma(\mu)(r)$ and proved \eqref{dz} for $n=2$. 

From the argument in subsection \ref{alter} we see that when $\mu$ is replaced by $\mu\times\nu$ we have

\begin{align*}
&\iint \underline{D}(S_{g\#}(\mu\times\nu))(z)^2\,dz\,d\theta_ng\\
&\leq\liminf_{r\to 0}\iint \a(n)^{-1}r^{-n}S_{g\#}(\mu\times\nu)(B(z,r))\,dS_{g\#}(\mu\times\nu) z\,d\theta_ng\\
=&\liminf_{r\to 0}\int \a(n)^{-1}r^{-n}\theta_n(\{g:|x-g(y)-(u-g(v))|\leq r\})\\
&d(\mu\times\nu) (x,y)\,d(\mu\times\nu) (u,v)\\
\leq&\liminf_{r\to 0}c\int r^{-1}\mu\times\mu(\{(x,u):||x-u|-|y-v||\leq r\})|y-v|^{1-n}\\
&d(\nu\times\nu) (y,v)\\
&= \liminf_{r\to 0}c\int r^{-1}\delta(\mu)(B(t,r))t^{1-n}\,d\delta(\nu)t\\
&=c\int \delta(\mu)(t)\delta(\nu)(t)t^{1-n}\,dt,
\end{align*}
provided the distance measures $\delta(\mu)$ and $\delta(\nu)$ are $L^2$ functions, and even a bit better so that we can move $\liminf$ inside the integral. In fact, we have equality everywhere in the above argument if $\mu$ and $\nu$ are smooth functions with compact support. Since by an example in \cite{GIOW}, when $n=2$, for any $s<4/3$,\  $I_{s}(\mu)<\infty$ is not enough for $\delta(\mu)$ to be in $L^2$, probably it is not enough for  $S_{g\#}(\mu\times\mu)$ to be in $L^2$. But  in \cite{GIOW} it was shown that if $I_{s}(\mu)<\infty$ for some $s>5/4$, there is a modification of $\mu$ with good $L^2$ behaviour.  Maybe this method could be used to show, for instance, that if $n=2$ and $\dim A = \dim B > 5/4$, then $\mathcal L^2(\mathcal S_g(A\times B))>0$ for almost all $g\in O(2)$.  One problem is that for distance sets one can split the measure to two parts with positive distance and only consider distances between points in the different supports, so one need not consider arbitrarily small distances, and the authors of \cite{GIOW} seem to use this essentially. Here such reduction may not be possible.

\subsection{Hausdorff dimension of intersections}
One motivation for this study is hope to shed light on intersection problems.  The main question is: what conditions on $\dim A$ and $\dim B$ guarantee that for almost all $g\in O(n)$, 
$\dim A\cap(g(B)+z) \geq \dim A + \dim B - n-\epsilon$ for positively many $z\in\R^n$ for every $\epsilon>0$. I expect that $\dim A +\dim B > n$ should be enough.  This is only known when one of the sets has dimension bigger than $(n+1)/2$. A necessary condition of course is that $\mathcal{L}^n(S_g(A\times B))>0$ for almost all $g\in O(n)$. By Theorem \ref{thm3} we have this when $\dim A+(n-1)\dim B/n>n$, but even then for general $A$ and $B$ I only know 
the estimate $\dim A\cap(g(B)+z) \geq \dim A + (n-1)\dim B/n - n-\epsilon$, which follows from Theorem 4.1 in \cite{M7} and \eqref{dz}. Since $A\cap(g(B)+z)$ is the projection on the first factor of $(A\times B)\cap S_g^{-1}(z)$, the problem is equivalent to getting dimension estimates for the sections $(A\times B)\cap S_g^{-1}(z)$. This point of view together with the results of Section \ref{product sets} is used in \cite{M8} to show that for almost all $g\in O(n)$, 
$\dim A\cap(g(B)+z) \geq \dim A + \dim B - n$ for positively many $z\in\R^n$ if $A$ and $B$ have positive and finite Hausdorff measures in their dimensions and $\liminf_{r\to 0}r^{-s}\mathcal H^s(A\cap B(x,r))>0$ for $\mathcal H^s$ almost all $x\in A$, where $s=\dim A$, and respectively for $B$.

Replacing $O(n)$ with the bigger group of similarities, maps $rg, r>0, g\in O(n),$ better results were obtained in \cite{K}, \cite{M1} and \cite{M7}. Then no condition like $\dim A>(n+1)/2$ is needed. In fact, Kahane considered the more general situation of closed subgroups of the general linear group which act transitively outside the origin.

\vspace{1cm}
\begin{footnotesize}
{\sc Department of Mathematics and Statistics,
P.O. Box 68,  FI-00014 University of Helsinki, Finland,}\\
\emph{E-mail address:} 
\verb"pertti.mattila@helsinki.fi" 

\end{footnotesize}

\end{document}